\documentclass[12pt]{amsart}
\usepackage[margin=1in]{geometry}
\usepackage{amsmath}
\usepackage{amsfonts}
\usepackage{amssymb}
\usepackage{graphicx}
\usepackage{mathrsfs}
\usepackage{graphicx,color}
\usepackage[mathscr]{eucal}
\usepackage{comment}

\usepackage{tikz}
\usepackage{tikz-cd}
\usepackage{graphicx}
\usepackage[shortlabels]{enumitem}
\usepackage{hyperref,url}
\hypersetup{colorlinks,linkcolor=blue,citecolor=blue}

\usetikzlibrary{decorations.pathmorphing}
\tikzset{snake it/.style={decorate, decoration=snake}}

\newtheorem{theorem}{Theorem}[section]
\newtheorem{lemma}[theorem]{Lemma}

\newtheorem{corollary}[theorem]{Corollary}

\theoremstyle{definition}
\newtheorem{definition}[theorem]{Definition}
\newtheorem{example}[theorem]{Example}

\theoremstyle{remark}
\newtheorem{remark}[theorem]{Remark}

\numberwithin{equation}{section}

\newcommand{\abs}[1]{\lvert#1\rvert}

\newcommand{\CC}{\mathbb{C}}
\newcommand{\ZZ}{\mathbb{Z}}
\newcommand{\QQ}{\mathbb{Q}}
\newcommand{\RR}{\mathbb{R}}

\newcommand{\Ham}{\mathrm{Ham}}

\newcommand{\Hofer}{\mathrm{Hofer}}

\begin{document}

\title{Spectral diameter of negatively monotone manifolds}

\author{Yuhan Sun}
\address{Department of Mathematics, Imperial College London, South Kensington, London, SW7 2AZ, UK}
\email{yuhan.sun@imperial.ac.uk}

\begin{abstract}
For a closed negatively monotone symplectic manifold, we construct quasi-isometric embeddings from the Euclidean spaces to its Hamiltonian diffeomorphism group, assuming it contains an incompressible heavy Lagrangian. Moreover, if it is integral, we show the super-heaviness of its skeleton with respect to a Donaldson hypersurface.
\end{abstract}

\maketitle

\tableofcontents

%% The correct journal style for \specialsection is all uppercase; a known bug
%% in amsart.cls prevents this, so input must be uppercase until it is fixed.

\section{Introduction}

Given a closed symplectic manifold $(M, \omega)$, the group of Hamiltonian diffeomorphisms $\Ham(M,\omega)$ is a central object of research in symplectic topology. This group is endowed with the Hofer norm
$$
\abs{\cdot}_\Hofer :\Ham(M,\omega)\to \RR_{\geq 0},
$$
by Hofer \cite{Hofer}, and the spectral norm
$$
v:\Ham(M,\omega)\to \RR_{\geq 0},
$$
by Schwarz \cite{Sc} and Oh \cite{Oh}. A main question related to the Hofer norm is whether it is an unbounded function for any closed symplectic manifold, and when it is, which unbounded group admits a quasi-isometric embedding to it. These are usually referred to as the Hofer diameter conjecture \cite[Problem 20]{MS} and the Kapovich-Polterovich question \cite[Problem 21]{MS} respectively. Similarly, one can ask the same question about the spectral norm.

In this article, we explore these questions for \textit{negatively monotone} symplectic manifolds. That is, $c_{1}(TM)\mid_{\pi_{2}(M)}=\tau\cdot\omega\mid_{\pi_{2}(M)}$ for some negative constant $\tau$.

A subset $K$ of $M$ is called incompressible if the inclusion-induced map $\pi_{1}(K)\to \pi_{1}(M)$ is injective. Particularly, any simply-connected subset $K$ is incompressible. Our main result is

\begin{theorem}\label{t:main}
    Let $(M^{2n},\omega)$ be a closed negatively monotone symplectic manifold with $n\geq 3$. If $M$ contains an incompressible heavy Lagrangian, then for any positive integer $N$ there is a group homomorphism
    $$
    I:(\RR^N, \abs{\cdot}_\infty)\to \Ham(M,\omega),
    $$
    satisfying
    $$
    \abs{x}_\infty\leq v(I(x))\leq \abs{I(x)}_{\Hofer}\leq  2\abs{x}_\infty.
    $$
\end{theorem}

Particularly, the map $I$ is a quasi-isometric embedding of $\RR^N$ and its image is called an $N$-dimensional \textit{quasi-flat}; see \cite[Section 1.1]{CHS} for more detailed discussions. It is known to exist under various topological or dynamical conditions \cite{Py,Usher}. A recent breakthrough \cite{CHS,PS} is the construction of infinite-dimensional quasi-flats inside $\Ham(S^2)$. 

Our method of constructing $I$ is to compute spectral invariants of functions supported in neighborhoods of incompressible Lagrangians and use Entov-Polterovich's theory of \textit{heavy} sets \cite{EP09}. 

Lagrangian Floer cohomology detects heavy Lagrangians \cite{Albers,BC,EP09,FOOO2019}. More precisely, a Lagrangian $L$ with non-zero Floer cohomology (possibly with bounding cochains) is heavy; see \cite[Theorem 1.6]{FOOO2019}. It enables us to present concrete applications.

\begin{example}\label{ex:asph}
    If $(M,\omega)$ is closed symplectically aspherical and it contains an incompressible Lagrangian, then $M$ admits quasi-isometric embeddings $I:(\RR^N, \abs{\cdot}_\infty)\to \Ham(M,\omega)$ for any $N$ when $\dim M\geq 6$.
\end{example}
\begin{proof}
    Recall $(M,\omega)$ is called symplectically aspherical if $c_{1}(TM)\mid_{\pi_{2}(M)}=\omega\mid_{\pi_{2}(M)}=0$. Hence it is negatively monotone for any negative constant $\tau$. An incompressible Lagrangian $L$ in $M$ has non-zero Floer cohomology because $\omega\mid_{\pi_2(M,L)}=0$.
\end{proof}

\begin{example}\label{ex:sphere}
    Let $(M,\omega)$ be a degree-$d$ smooth hypersurface in $\CC P^{n+1}$ with the Fubini-Study symplectic structure. If $d> n+2$, then $M$ admits quasi-isometric embeddings $I:(\RR^N, \abs{\cdot}_\infty)\to \Ham(M,\omega)$ for any $N$.
\end{example}
\begin{proof}
    When $d> n+2$, a degree-$d$ smooth hypersurface is a negatively monotone symplectic manifold with respect to the induced Fubini-Study structure. It always contains a Lagrangian sphere by the vanishing cycle construction. A Lagrangian sphere in a negatively monotone manifold has non-zero Floer cohomology; see Section \ref{sec:sphere} for more discussions and examples.
\end{proof}

\begin{example}
     Let $(M,\omega)$ be a closed negatively monotone symplectic manifold. Then the product $(M\times M, -\omega\times\omega)$ admits quasi-isometric embeddings $I:(\RR^N, \abs{\cdot}_\infty)\to \Ham(M,\omega)$ for any $N$.
\end{example}
\begin{proof}
    Note that $(M\times M, -\omega\times\omega)$ contains an incompressible Lagrangian diagonal which has non-zero Floer cohomology by \cite[Theorem 1.9]{FOOO2017} and \cite[Theorem 1.2]{CW}.
\end{proof}

Next we give more details about backgrounds and our proofs, as well as several results in more general settings.

\subsection{Complements of Lagrangian submanifolds}

For a symplectic manifold $(M,\omega)$, we say that it is non-positive if it satisfies the following condition
\begin{equation}\label{eq:condition}
    \forall B\in\pi_2(M), \quad \text{if} \quad c_1(TM)(B)<0 \quad \text{then} \quad \omega(B)> 0.
\end{equation}
Any symplectic Calabi-Yau manifold (meaning $c_1(TM)\mid_{\pi_2(M)}=0$) or any negatively monotone symplectic manifold satisfies this condition. We say $(M,\omega)$ is \textit{rational} if 
$$
\{\omega(B)\mid B\in\pi_2(M)\}\subset \RR
$$
is a discrete subgroup. A negatively monotone manifold is always rational since $c_1$ takes discrete values on $\pi_2(M)$ and the monotonicity constant $\tau$ is strictly negative.

The theory of spectral invariants provides a way to study Hamiltonian groups via Floer theory. Associated with any non-zero class $A$ in the quantum cohomology $QH(M,\omega)$ of $M$ and a time-dependent Hamiltonian function $H$, there is a number $c(A,H)\in\RR$, called the spectral invariant of $A$ and $H$; see Section \ref{sec:si} for the definition. We will mainly focus on the case that $A=1\in QH^0(M,\omega)$ being the unit of the quantum cohomology. The spectral norm of $H$ is defined as
\begin{equation}\label{eq:spectral norm}
    V(H):= -c(1,H)-c(1,\bar{H}),
\end{equation}
where $\bar{H}(t,x):=-H(t,\phi^t_H(x))$ and $\phi^t_H$ is the time-$t$ Hamiltonian flow of $H$. The homotopy invariance property of $c$ gives a well-defined map
$$
V:\widetilde{\Ham}(M,\omega)\to\RR_{\geq 0}, \quad V(\phi^t_H):=V(H)
$$
on the universal cover $\widetilde{\Ham}(M,\omega)$ of $\Ham(M,\omega)$. This function $V$ is only a pseudo-norm. By taking infimum among all deck transformations, one defines the spectral norm as
\begin{equation}
    v:\Ham(M,\omega)\to\RR_{\geq 0}, \quad v(\psi):= \inf_{\psi=\phi^1_H}V(H),
\end{equation}
which is a conjugation invariant norm \cite{Oh,Sc}.

\begin{remark}
    The definition of spectral invariants is sensitive to the ground coefficient of the quantum cohomology. Our proofs work for any coefficient whenever the Hamiltonian Floer theory is defined. Due to our sign conventions, there are two negative signs in (\ref{eq:spectral norm}).
\end{remark}

It is a natural problem to study the diameter of $\Ham(M,\omega)$ equipped with the norm $v$. Entov-Polterovich \cite{EP03} proved the spectral diameter of $\CC P^n$ is finite when the ground ring of $QH(\CC P^n, \omega_{FS})$ is a field. On the other hand, Kawamoto-Shelukhin \cite{KS} proved the spectral diameter of $\CC P^n$ is infinite when the ground ring is $\ZZ$. When $M$ is symplectically aspherical, Mailhot \cite{Mailhot} gave one sufficient condition for the spectral diameter to be infinite. In general the spectral diameter is poorly understood. For example, it is not known whether it is always infinite when $QH(M,\omega)$ is very far from being a field \cite[Remark 1.11]{McDuff}. Recall that the quantum cohomology ring \cite{MS2} is a deformation of the classical cohomology ring, which is never a field for any closed manifold. In order for $QH(M,\omega)$ to be a field, one usually expects $M$ to have lots of holomorphic curves to deform the usual cup product. In this article, we consider negatively monotone manifolds, which are expected to have fewer non-trivial holomorphic curves.

Recall that, by Entov-Polterovich \cite{EP09}, a compact subset $K$ of $M$ is called \textit{heavy} if 
$$
c(1,H)=0, \quad \forall H\geq 0, H\mid_K =0
$$
and it is called \textit{super-heavy} if
$$
\mu(H):=\lim_{m\to+\infty} \dfrac{c(1,mH)}{m}=0, \quad \forall H\leq 0, H\mid_K =0.
$$
It is enough to just use autonomous Hamiltonian functions in this definition; see Section \ref{sec:si} for our conventions on the definitions. These notions reveal surprising rigidity phenomena of compact subsets of a symplectic manifold.\footnote{Generally a (super-)heavy set is defined with respect to an idempotent of the quantum cohomology. In this article we only consider the unit.} They are helpful to us in that if there exist $N+1$ disjoint heavy sets then there is a group homomorphism
    $$
    I:(\RR^N, \abs{\cdot}_\infty)\to C^\infty(M)
    $$
    satisfying
    $$
    \abs{x}_\infty\leq V(I(x))\leq \abs{I(x)}_{\Hofer}\leq  2\abs{x}_\infty;
    $$
see Lemma \ref{l:unbounded}. This reduces the task of constructing quasi-isometric embeddings to finding disjoint heavy sets. Suppose we already have a heavy Lagrangian $L$, to find a second heavy set disjoint from $L$, it is natural to look at the complement of $L$.

For any closed Lagrangian submanifold $L$ of $(M,\omega)$, a neighborhood $U$ of it is called Weinstein if it is symplectomorphic to 
$$
U_{1,g}T^{*}L:= \lbrace (q, p)\in T^{*}L : \abs{p}_{g}<1 \rbrace
$$
for some metric $g$ on $T^*L$ that is equipped with the canonical symplectic form.

\begin{theorem}\label{t:main lag}
	Let $(M^{2n}, \omega)$ be a closed rational symplectic manifold satisfying Condition (\ref{eq:condition}) with $n\geq 3$. Let $L$ be a closed Lagrangian submanifold of $M$. If $L$ is incompressible, then $M-U$ is heavy for any Weinstein neighborhood $U$ of $L$.
\end{theorem}

In order to get more disjoint heavy sets, we consider the boundaries of the Weinstein neighborhoods. Note that by varying the size of $U$, there is a family of disjoint contact hypersurfaces, parameterized by an open interval.  A technical enhancement of Theorem \ref{t:main lag} is the following. 

\begin{theorem}\label{t:boundary}
    Let $(M^{2n}, \omega)$ be a closed rational symplectic manifold satisfying Condition (\ref{eq:condition}) with $n\geq 3$. Let $L$ be a closed Lagrangian submanifold of $M$. If $L$ is incompressible and heavy, then $\partial(M-U)$ is heavy for any Weinstein neighborhood $U$ of $L$. 
\end{theorem}

\begin{remark}
	\begin{enumerate}
 
		\item Theorem \ref{t:main lag} implies that $L$ is not super-heavy in the above setting, since a super-heavy set intersects any heavy set. On the other hand, many examples of super-heavy Lagrangians are known; see Subsection 1.6 in \cite{EP09}.

        \item Theorem \ref{t:boundary} is a natural analogue of two facts: an incompressible annulus in a torus is foliated by a family of heavy circles; and the rigidity of a Liouville domain implies the rigidity of its boundary \cite{CFO,Ritter}.

	\end{enumerate}
\end{remark}

McDuff (Theorem 1.3 in \cite{McDuff}) and Kawamoto (Theorem 6 in \cite{K}) give lots of conditions that the pseudo-norm $V$ descends to the norm $v$. Particularly, it is true when $(M,\omega)$ is negatively monotone. Therefore Theorem \ref{t:main} follows from Theorem \ref{t:boundary} and Lemma \ref{l:unbounded}. Moreover, in the symplectic Calabi-Yau setting we have  

\begin{corollary}\label{co:diam}
	If a closed rational symplectic Calabi-Yau manifold $M^{2n}$ with $n\geq 3$ contains an incompressible heavy Lagrangian, suppose further that
	\begin{enumerate}
		\item the even degree homology of $M$ is generated as a ring by $H_{2n-2}(M)$, or;
		\item the quantum product on $QH^{*}(M)$ is the usual cup product on $H^{*}(M)$;
	\end{enumerate}  
    then for any positive integer $N$ there is a group homomorphism
    $$
    I:(\RR^N, \abs{\cdot}_\infty)\to \Ham(M,\omega)
    $$
    satisfying
    $$
    \abs{x}_\infty\leq v(I(x))\leq \abs{I(x)}_{\Hofer}\leq  2\abs{x}_\infty.
    $$
\end{corollary}

Note that any K\"ahler 6-manifold $M$ satisfies that the even degree homology is generated as a ring by $H_{4}(M)$. Hence sample examples of Corollary \ref{co:diam} are projective Calabi-Yau 6-manifolds which, moreover, contain Lagrangian spheres. Examples satisfying $(2)$ in Corollary \ref{co:diam} include symplectically aspherical manifolds. In this setting see related results \cite{Mailhot,MVZ}. Beyond the symplectically aspherical case, hyper-K\"ahler manifolds also have undeformed quantum products; see Section \ref{sec:ud} for more discussions.

\begin{remark}
    \begin{enumerate}
        \item We assume that $(M,\omega)$ is closed but we expect our results can be extended to symplectic manifolds which are convex at infinity. Particularly, cotangent bundles should fit in Example \ref{ex:asph}.
        \item Besides Lagrangian Floer cohomology, the symplectic cohomology with support is also closely related to heavy sets \cite{MSV}. It is not used in this article and we hope to explore its relation with the spectral norm in the future.
    \end{enumerate}
\end{remark}

\subsection{Super-heavy Lagrangian skeleton}
Our method of computing spectral invariants is also applicable to radius-type functions on symplectic disk bundles. We present one application here about the symplectic rigidity of certain Lagrangian skeleta. A symplectic manifold $(M,\omega)$ is called \textit{integral} if the symplectic class $[\omega]$ admits a lift to $H^2(M;\ZZ)$. Given a closed integral symplectic manifold $(M,\omega)$ and a Donaldson hypersurface $Y$ in $M$, there is a Biran-Giroux decomposition 
$$
(M, \omega)\cong (U, \omega_U)\sqcup L_Y
$$
where $U$ is a unit disk bundle over $Y$ with a standard symplectic form $\omega_U$ and $L_Y$ is an isotropic CW complex, called the skeleton \cite{Biran,Gir}. We remind the reader that the construction of the pair $(U, L_Y)$ depends on the choices of metrics and connections. In particular, $L_Y$ is not determined merely by $Y$ and the following result works for $L_Y$ with respect to any such choices; see Remark \ref{re:comp}.

\begin{theorem}\label{t:skeleton}
    Let $(M^{2n},\omega)$ be a closed integral symplectic manifold satisfying Condition (\ref{eq:condition}). Let $Y^{2n-2}$ be a closed symplectic hypersurface in $M$ whose Poincar\'e dual is $k[\omega]$ for some positive integer $k$. The skeleton $L_Y$ is super-heavy. 
\end{theorem}

Various rigidity phenomenon of the skeleta was discovered since its introduction in \cite{Biran}. Recently it is discussed by \cite{BSV,TVar,MSV} in positively monotone and symplectic Calabi-Yau case. Here our Condition (\ref{eq:condition}) provides some complementary cases. A standard corollary is

\begin{corollary}
    Let $(M^{2n},\omega)$ be a closed integral symplectic manifold satisfying Condition (\ref{eq:condition}). Let $Y^{2n-2}$ be a closed symplectic hypersurface in $M$ whose Poincar\'e dual is $k[\omega]$ for some positive integer $k$. There hold
    \begin{enumerate}
        \item $L_Y$ is not displaceable from itself by any symplectomorphism. Particularly, it has an $n$-dimensional cell.
        \item Any Lagrangian in $M$ with non-zero Floer cohomology is not displaceable from $L_Y$ by any Hamiltonian diffeomorphism.
    \end{enumerate}
\end{corollary}
\begin{proof}
    This follows from the properties of heavy and super-heavy sets \cite[Theorem 1.4]{EP09}. Note that if $L_Y$ does not have any $n$-dimensional cell then it is Hamiltonian displaceable by \cite[Lemma 3.2]{BiCi}.
\end{proof}

\section{Floer theory backgrounds}
In this section we briefly review the Hamiltonian Floer theory and spectral invariants. For symplectic Calabi-Yau manifolds, classical analysis \cite{Floer,HS} of moduli spaces is enough to define Floer cohomology. If the symplectic class belongs to $H^2(M;\QQ)$ then classical transversality argument also works as developed in \cite{CM,CW}. For general situations, we need to use virtual techniques, like in \cite{FOOO2019, Pardon, AMS}. Our proofs use analysis of actions and indices of Hamiltonian orbits, which are independent from the problem of regularity of moduli spaces. Hence, we generally omit those discussions here. We refer to \cite{V,Sc,Oh,FOOO2019} for the full treatment of spectral invariants in Hamiltonian Floer theory.

\subsection{Hamiltonian Floer theory}
We fix a closed symplectic manifold $(M, \omega)$. Let $H_{t}: \RR/\ZZ\times M\to\RR$ be a time-dependent non-degenerate Hamiltonian function on $M$ and let $\mathcal{P}_{H_{t}}$ be the set of contractible one-periodic Hamiltonian orbits of $H_{t}$. Our convention on the Hamiltonian vector field is $\omega(X_H,\cdot)= dH$. For an orbit $\gamma\in \mathcal{P}_{H_{t}}$, a capping is a smooth map $u$ from a disk to $M$ such that it equals $\gamma$ when restricted to the boundary on the disk. For two cappings $u, w$ of the same orbit $\gamma$, we say they are equivalent if
$$
\int u^*\omega= \int w^*\omega, \quad \text{and} \quad \int u^*c_1(TM)= \int w^*c_1(TM).
$$
Then we write $\widetilde{\mathcal{P}_{H_{t}}}$ as the set of equivalence classes of capped orbits. A general element is written as $[\gamma,u]$. Its action is defined by
\begin{equation}
	\mathcal{A}_{H_{t}}([\gamma, u])= \int_{\gamma}H_{t} +\int u^{*}\omega,
\end{equation}
and its degree is defined by
\begin{equation}
	\deg([\gamma, u])= n+ CZ([\gamma, u])
\end{equation}
where $CZ([\gamma, u])$ is the Conley-Zehnder index of this capped orbit and $n$ is half of the real dimension of $M$. We refer to \cite[Section 2.4]{Salamon} for the definition of the Conley-Zehnder index. If $H$ is a $C^2$-small Morse function, a critical point $p$ can be viewed as a Hamiltonian orbit with a constant capping $[p,c]$. Our convention is that $\deg([p,c])$ is the Morse index of $p$.

Given any class $B\in\pi_2(M)$ and a capped orbit $[\gamma,u]$, we can form a new capped orbit $[\gamma,u+B]$ where $u+B$ is given by gluing a sphere representing $B$ to $u$. The index and action change as follows.
\begin{equation}\label{eq:capchange}
    \deg([\gamma, u+B])=\deg([\gamma, u])+ 2c_1(TM)(B), \quad \mathcal{A}_{H_{t}}([\gamma, u+B])= \mathcal{A}_{H_{t}}([\gamma, u])+ \omega(B).
\end{equation}

Then we consider the $\QQ$-vector space generated by formal sums of degree-$k$ capped orbits, and define
$$
CF^{k}(H_{t}):= \lbrace x=\sum_{i=1}^{+\infty} a_{i}[\gamma_{i}, u_{i}]\mid a_{i}\in\QQ, \deg([\gamma_{i}, u_{i}])=k, \lim_{i\rightarrow +\infty} \mathcal{A}_{H_{t}}([\gamma_{i}, u_{i}])=+\infty\rbrace.
$$
The action of a linear combination of capped orbits is defined as
$$
\mathcal{A}_{H_t}(x):= \min\{\mathcal{A}_{H_t}([\gamma_i, u_i])\mid a_i\neq 0\},
$$
which we call the $\min$-action. We define the degree-$k$ action spectrum as follows.
\begin{equation}
    Spec_k(H_t):= \{\mathcal{A}_{H_t}(x)\mid x\in CF^k(H_t)\}.
\end{equation}
If $H_t$ is non-degenerate then $Spec_k(H_t)$ is a countable subset of $\RR$. Moreover, if $(M,\omega)$ is rational then $Spec_k(H_t)$ is a discrete subset of $\RR$.

There is a differential, defined by counting Floer cylinders,
$$
d: CF^{k}(H_{t})\rightarrow CF^{k+1}(H_{t})
$$
such that $d^{2}=0$, which gives the Hamiltonian Floer cohomology $HF^{*}(H_t)$ of $M$. We remark that in our convention the differential \textit{increases} the degree and the action. A Floer cylinder sends the orbit at negative infinity to the orbit at positive infinity. Hence we have a filtration on the Floer complex
\begin{equation}
	CF^{k}_{\geq p}(H_{t}):= \lbrace x\in CF^{k}(H_{t})\mid \mathcal{A}(x)\geq p\rbrace
\end{equation}
for any real number $p$. 

The ground coefficient ring $\QQ$ is used for defining the differential on general symplectic manifolds. One can use integers for symplectic Calabi-Yau manifolds.

\subsection{Spectral invariants}\label{sec:si}

Now we review the definitions of spectral invariants and heavy sets. For a symplectic manifold $(M,\omega)$, its quantum cohomology is
$$
QH^*(M,\omega; \Lambda_\omega):= H^*(M;\QQ)\otimes_\QQ\Lambda_\omega,
$$
where $\Lambda_\omega$ is a suitable Novikov ring. We refer to \cite{MS2} for more details. In the remainder of this article, we simply write $QH^*(M,\omega; \Lambda_\omega)$ as $QH^*(M)$.

Consider the PSS isomorphism \cite{PSS,FOOO2019}
$$
PSS_H: QH^*(M)\to HF^*(H).
$$
We have the \textit{cohomological spectral invariants} for any non-zero degree-$k$ class $A\in QH^*(M)$
$$
c(A, H):=\sup \lbrace \mathcal{A}_H(x)\in \mathbb{R}\mid x\in CF^k(H), dx=0, [x]=PSS_H(A)\rbrace.
$$
The spectral invariants enjoy several good properties; see \cite{FOOO2019}. 

\begin{theorem}
	Let $A$ be a non-zero class of pure degree in $QH^*(M)$, and let $H$ be a non-degenerate Hamiltonian function. The number $c(A, H)$ has the following properties.
	\begin{enumerate}
	\item (Finiteness) $c(A, H)$ is a finite number in $Spec_k H$ where $k=\deg A$.
	\item (Hamiltonian shift) $c(A, H+\lambda(t))= c(A, H)+ \int_{S^{1}}\lambda(t)$ for any function $\lambda(t)$ on $S^{1}$.
	\item (Lipschitz property) $\int_{S^{1}}\min_{M} (H_{1}- H_{2})\leq c(A, H_{1})- c(A, H_{2})\leq \int_{S^{1}} \max_{M} (H_{1}- H_{2})$. In particular, if $H_{1}\geq H_{2}$ then $c(A, H_{1})\geq c(A, H_{2})$.
	\item (Triangle inequality) $c(A_{1}*A_{2}; H_{1}\# H_{2})\geq c(A_{1}, H_{1})+ c(A_{2}, H_{2})$ where $*$ is the quantum product.
	\item (Homotopy invariance) $c(A, H_{1})= c(A, H_{2})$ for any two normalized Hamiltonian functions generating the same element in $\widetilde{\Ham}(M)$.
	\end{enumerate}
\end{theorem}

The Lipschitz property enables us to define $c(A,H)$ for continuous $H$ using approximations. Particularly, it tells us that $c(A,H=0)=0$ and $\min H\leq c(A,H)\leq \max H$.

Let $1\in QH^{0}(M)\cong H^{0}(M)$ be the unit of the quantum cohomology ring. Then its corresponding partial symplectic quasi-state $\mu: C^0(M)\to \RR$ is defined as
$$
\mu(H):= \lim_{m\rightarrow \infty} \frac{c(1, mH)}{m}.
$$ 
Then we can define the (super)heaviness.

\begin{definition}
	For a compact set $K$ in $M$,
	\begin{enumerate}
		\item it is heavy if $\mu(H)\leq \max_{K} H$ for any smooth function $H$ on $M$;
		\item it is super-heavy if $\mu(H)\geq \min_{K} H$ for any smooth function $H$ on $M$.
	\end{enumerate}
\end{definition}

\begin{remark}
	 The above definition is different from, but equivalent to the original definition in \cite{EP09} since the sign conventions are different. One can verify it by using the duality formula
	$$
	-c(1, -H)= c^{EP}([M], H),
	$$
 see Section 4.2 in \cite{LZ}. Here $[M]\in H_{2n}(M)$ is the fundamental class.
\end{remark}

We have several handy criteria for the (super)heaviness.

\begin{lemma}
	For a compact set $K$ in $M$,
	\begin{enumerate}
		\item it is heavy if and only if for any non-negative function on $M$ that is zero on $K$, we have $\mu(H)=0$;
		\item it is heavy if and only if for any non-negative function on $M$ that is zero on $K$, we have $c(1, H)=0$.
        \item it is super-heavy if and only if for any non-positive function on $M$ that is zero on $K$, we have $\mu(H)=0$.
	\end{enumerate}
\end{lemma}
\begin{proof}
	Up to sign difference, $(1)$ and $(3)$ are proved in \cite{EP09}. So we only show $(2)$.
	 
	If for any non-negative function on $M$ that is zero on $K$ we have $c(1, H)=0$, then $c(1, kH)=0$ for any positive integer $k$. Hence $\mu(H)= 0$ for all such functions. By $(1)$ we know $K$ is heavy. 
	
	On the other hand, by the triangle inequality we know that $c(1,mH)\geq mc(1,H)$ for any positive integer $m$. Hence $\mu(H)\geq c(1,H)$. If $K$ is heavy, then $(1)$ tells us that for any non-negative function on $M$ that is zero on $K$, $\mu(H)=0$. So we get $0=\mu(H)\geq c(1,H)\geq 0$, where the last inequality is the monotonicity of spectral invariants. Hence $c(1,H)=0$ for all such functions.
\end{proof}

\begin{comment}

\begin{lemma}\label{l:ind spec}
    Let $H_{s,t}:[0,1]\times S^1\times M\to \RR$ be a continuous family of Hamiltonian functions. Suppose that there exists $\epsilon >0$ and a finite subset $I$ of $(0,1)$ such that 
    $$
    Spec_{k}(H_{s,t})\cap (\min H_{0,t}-\epsilon, \max H_{0,t}+\epsilon)
    $$ 
    is independent of $s$ whenever $s\in [0,1]-I$, and is a countable subset of $\RR$, then $c(A, H_{s,t})$ is independent of $s$ for any degree-$k$ class $A$.
\end{lemma}
\begin{proof}
    By the Lipschitz continuity of spectral invariants, the function $\rho(s):=c(A, H_{s,t})$ is a continuous function from $[0,1]$ to $\RR$. For $s>0$ small enough, we have 
    $$
    \rho(s)\in Spec_{k}(H_{s,t})\cap (\min H_{0,t}-\epsilon, \max H_{0,t}+\epsilon).
    $$
    A continuous function with a countable image can only be a constant function. Hence the set $\{s\mid \rho(s)=\rho(0)\}$ is an open subset of $[0,1]-I$. On the other hand, it is a closed subset by definition. Hence $\rho(s)$ is constant on the component of $[0,1]-I$ containing $0$. Since $I$ is finite and $\rho$ is continuous, we can extend the argument and show $\rho$ is constant on $[0,1]$.
\end{proof}

\end{comment}

Given the definition of $c(1,H)$, one defines 
$$
V(H):= -c(1,H)-c(1,\bar{H})
$$
and 
$$
v:\Ham(M,\omega)\to \RR, \quad v(\phi):=\inf_{\phi=\phi^1_H}V(H).
$$
We say $V$ descends to $v$ if $v(\phi)=V(H)$ for any $H$ with $\phi=\phi_H$. For example, this is true when $(M,\omega)$ is negatively monotone by \cite[Theorem 6]{K}; also see \cite[Lemma 3.2]{McDuff}. Recall the Hofer norm is defined as
$$
\abs{H}_\Hofer:= \int_0^1 (\max_M H(t,\cdot)- \min_M H(t,\cdot))dt
$$
and
$$
\abs{\cdot}_\Hofer: \Ham(M,\omega)\to \RR, \quad \abs{\phi}_\Hofer:= \inf_{\phi=\phi^1_H}\abs{H}_\Hofer.
$$
It is known that $V(H)\leq \abs{H}_\Hofer$ and $v(\phi)\leq \abs{\phi}_\Hofer$.

Let $(M, \omega)$ be a closed symplectic manifold and $Z$ be a compact smooth hypersurface. Suppose a neighborhood of $Z$ is diffeomorphic to $(1-\epsilon,1+\epsilon)\times Z$, we write $Z_b:=\{ b\}\times Z$. 

\begin{lemma}\label{l:unbounded}
     Suppose $V$ descends to $v$. If $Z_b$ is heavy for any $b\in (1-\epsilon,1+\epsilon)$, then there is a group homomorphism 
     $$
     I:(\RR^N, \abs{\cdot}_\infty)\to \Ham(M,\omega)
     $$
     satisfying
     $$
     \abs{x}_\infty\leq v(I(x))\leq \abs{I(x)}_\Hofer\leq  2\abs{x}_\infty
     $$
     for any positive integer $N$. Here $\abs{\cdot}_\infty$ is the infinity norm
     $$
     \abs{x=(x_1,\cdots,x_N)}_\infty := \max_{1\leq i\leq N} \abs{x_i}.
     $$
\end{lemma}
\begin{proof}
    We will construct an embedding for $N=2$. The construction for general $N$ is similar.

    Consider a function $F:M\to \RR$ such that
    \begin{enumerate}
        \item $F$ is supported in $(1-\epsilon,1)\times Z$.
        \item $0\leq F\leq 1$ and $F=1$ in $(1-2\epsilon/3,1-\epsilon/3)\times Z$.
    \end{enumerate}
    Similarly, consider a function $G:M\to \RR$ such that
    \begin{enumerate}
        \item $G$ is supported in $(1,1+\epsilon)\times Z$.
        \item $0\leq G\leq 1$ and $G=1$ in $(1+\epsilon/3,1+2\epsilon/3)\times Z$.
    \end{enumerate}
    We define a map 
    $$
    I:\RR^2\to C^\infty(M), \quad I(s,t):= sF+tG.
    $$
    Next we compute the spectral invariants of $I(s,t)$.

    If $s=t=0$ then $c(1,0)=0$ and $V(0)=0$.

    If $s,t>0$ then $c(1,sF+tG)=0$ since $sF+tG$ is non-negative and $(sF+tG)^{-1}(0)$ is a heavy set. For example, it contains a heavy set $Z_1$. On the other hand, $-sF-tG+\max(s,t)$ is non-negative and its zero locus is a heavy set. For example, it contains either $Z_{1+\epsilon/2}$ or $Z_{1-\epsilon/2}$. So we have $c(1,-sF-tG)=-\max(s,t)$ and $V(sF+tG)= \max(s,t)= \abs{(s,t)}_\infty$. Similarly when $s,t<0$ we have $V(sF+tG)= \max(-s,-t)= \abs{(s,t)}_\infty$.

    If $s>0,t<0$, then $sF+tG-t$ is non-negative and its zero locus contains a heavy set $Z_{1+\epsilon/2}$. We get $c(1,sF+tG)=t$. Similarly, $c(1,-sF-tG)=-s$ and hence $V(sF+tG)=-t+s$, which is between $\abs{(s,t)}_\infty$ and $2\abs{(s,t)}_\infty$. If $s<0,t>0$ the proof is the same. In conclusion, we get $\abs{(s,t)}_\infty\leq V(I(s,t))\leq 2\abs{(s,t)}_\infty$.

    Next we consider the map
    $$
    I:\RR^2\to \Ham(M,\omega), \quad I(s,t):=\phi^1_{sF}\circ \phi^1_{tG}.
    $$
    Since $F,G$ are time-independent and have disjoint supports, their flow commute. Hence $I$ is a group homomorphism and $\phi^1_{sF+tG}=\phi^1_{sF}\circ \phi^1_{tG}$. We assume that $V$ descends to $v$, which means that $v(\phi^1_{sF}\circ \phi^1_{tG})= V(sF+tG)$. Therefore we get
    $$
    \abs{(s,t)}_\infty\leq V(sF+tG)= v(\phi^1_{sF}\circ \phi^1_{tG})\leq \abs{\phi^1_{sF}\circ \phi^1_{tG}}_\Hofer.
    $$
    On the other hand, we have 
    $$
    \abs{\phi^1_{sF}\circ \phi^1_{tG}}_\Hofer\leq \abs{sF+tG}_\Hofer\leq \abs{s}+\abs{t}\leq 2\abs{(s,t)}_\infty.
    $$
    This completes the proof for $N=2$. The proof for general $N$ is the same by picking $N$ functions with disjoint supports in $(1-\epsilon,1+\epsilon)\times Z$.

\end{proof}

This lemma uses that $V$ descends to $v$ in an essential way. Hence our theorems are restricted to the negatively monotone case, rather than just assuming Condition (\ref{eq:condition}). It would be very interesting to see how other invariants could be applied in our setup. For example, the boundary depth in \cite{Usher, KiS}.

\subsection{Undeformed quantum product}\label{sec:ud}
We record here another example in which $V$ may descend to $v$. Let $(M,\omega)$ be a closed symplectic manifold. Recall that the quantum product is defined by the 3-point genus zero Gromov–Witten invariants
$$
GW^{M,\omega}_{3,B}(A_1,A_2,A_3), \quad A_i\in H^*(M), \quad B\in H_2(M).
$$
We say the quantum product in $QH^*(M)$ is undeformed if all these invariants are zero unless $B=0$. If the quantum product is undeformed, McDuff \cite{McDuff} gave many conditions for $V$ to descend to $v$.

The following lemma is given by Jonny Evans in his answer to Question 106791 on the Mathoverflow website\footnote{https://mathoverflow.net/questions/106791/why-are-gromov-witten-invariants-of-k3-surfaces-trivial}. We record it here for the convenience of the readers.

\begin{lemma}
    Suppose that there is a smooth family of symplectic forms $\{\omega_t\}_{t\in[0,1]}$ on $M$ such that $\omega_0=-\omega_1$. Then we have the vanishing of Gromov-Witten invariants
    $$
    GW^{M,\omega_0}_{n,B}(A_1,\cdots,A_n)=0, \quad \forall A_i\in H^*(M), \quad \forall B\neq 0\in H_2(M).
    $$
\end{lemma}
\begin{proof}
    Suppose that $GW^{M,\omega_0}_{n,B}(A_1,\cdots,A_n)\neq 0$ for some $B\neq 0$. The deformation invariance property \cite[Remark 7.1.11]{MS2} shows that 
    $$
    GW^{M,\omega_0}_{n,B}(A_1,\cdots,A_n)= GW^{M,\omega_1}_{n,B}(A_1,\cdots,A_n)\neq 0.
    $$
    This implies that $\omega_0(B)>0$ and $\omega_1(B)>0$ by the definition of Gromov-Witten invariants, which is a contradiction to $\omega_0=-\omega_1$.
\end{proof}

The proof in \cite[Remark 7.1.11]{MS2} requires that all $(M,\omega_t)$ are semi-positive and expects the general case to be true via virtual techniques. Now suppose $(M,\omega)$ is a closed hyper-K\"ahler manifold. Then $\omega$ fits into a family of symplectic forms $\omega_t$ parameterized by the two-sphere $S^2$. In particular, $\omega$ and $-\omega$ can be connected by a path of symplectic forms. The first Chern class $c_1$ is zero with respect to any $\omega_t$. So we are in the semi-positive case and the above lemma tells us that the quantum product of $(M,\omega)$ is undeformed.

\section{Complements of Liouville domains}

A Liouville domain is a pair $(W,\lambda)$ where $W$ is a compact manifold with boundary and $\lambda$ is a one-form on $W$ such that
\begin{enumerate}
    \item $d\lambda$ is a symplectic form on $W$;
    \item the Liouville vector field $Y$ determined by $d\lambda(Y,\cdot)=\lambda$ is outward-pointing along $\partial W$. Particularly, the one-form $\alpha:=\lambda\mid_{\partial W}$ is contact.
\end{enumerate}
The symplectic completion of a Liouville domain $(W,\lambda)$ is a non-compact manifold 
$$
\widehat{W}:= W\cup_{\partial W} ([1,\infty)\times \partial W)
$$
equipped with a symplectic form $\omega$ which agrees with $d\lambda$ on $W$ and $\omega:= d(r\alpha)$ on the cylindrical region.

Given a contact manifold $(C,\xi=\ker\alpha)$ with a contact form $\alpha$, the Reeb vector field $R_\alpha$ on $C$ is the unique vector field satisfying
$$
\alpha(R_\alpha)=1, \quad d\alpha(R_\alpha,\cdot)=0.
$$
A Reeb orbit of period $T$ is a smooth map $\gamma: [0,T]\to C$ with $\gamma'(t)=R_\alpha(\gamma(t))$ and $\gamma(0)=\gamma(T)$. It is called non-degenerate if the linear return map, restricted to $\xi$, has no eigenvalues equal to one. A contact form is called non-degenerate if all of its closed Reeb orbits are non-degenerate. The action spectrum of $\alpha$ is the set of periods of its Reeb orbits. In the non-degenerate case, the action spectrum is a discrete subset of $\RR$.

\begin{example}\label{ex:Reeb}
    Let $L$ be a closed smooth manifold with a Riemannian metric $g$. Let $T^*L$ be the cotangent bundle of $L$ with the canonical Liouville one-form $\lambda$. Then the hypersurface 
    $$
    C= C_{1,g}T^*L:= \{(q,p)\in T^*L: \abs{p}_g=1\}
    $$
    is a contact manifold with a contact form $\alpha=\lambda\mid_C$. Any Reeb orbit on $C$ projects to a closed geodesic on $L$. Particularly, let $\gamma$ be a contractible Reeb orbit on $C$, the Conley-Zehnder index of $\gamma$ (computed in $C$) equals the Morse index of its underlying geodesic, which are non-negative. We learned this fact from \cite[Remark 1.16]{CFO} and the referee kindly points out a proof in \cite[Theorem 1.2]{Weber}. Note that if $\dim L\geq 3$ then any loop in $C$ that is contractible in $T^*L$ is also contractible in $C$.
\end{example}

Next we recall the relation between Reeb orbits and Hamiltonian orbits. Given a contact manifold $(C,\xi=\ker\alpha)$ and consider the symplectic manifold
$$
[1,\infty)\times C, \quad \omega= d(r\alpha).
$$
For a Hamiltonian $H(r,x)=f(r)$ which only depends on $r$, its Hamiltonian vector field is
$$
X_H(r,x)= -f'(r)R_\alpha.
$$
Particularly, a one-periodic orbit $\tilde{\gamma}$ of $X_H$ corresponds to a $\abs{f'(r_0)}$-periodic Reeb orbit $\gamma$ contained in the level set $\{r=r_0\}$ whenever $\abs{f'(r_0)}$ is in the action spectrum of $\alpha$. Given a trivialization of $\xi\mid_\gamma$, it induces a trivialization of $T([1,\infty)\times C)\mid_{\tilde{\gamma}}$. With respect to these trivializations, the Conley-Zehnder indices are related as
$$
\abs{CZ(\gamma) + CZ(\tilde{\gamma})} \leq 1, \quad f'(r)>0 \quad \text{and} \quad \abs{CZ(\gamma) - CZ(\tilde{\gamma})} \leq 1, \quad f'(r)<0.
$$
Note that if $f'(r)>0$ then $\tilde{\gamma}$ and $\gamma$ are in the opposite direction. If $\gamma$ is non-degenerate as a Reeb orbit, then $\tilde{\gamma}$ is \textit{transversally non-degenerate}, see \cite[Page 32]{CFHW}. The Conley-Zehnder index of $\tilde{\gamma}$ should be understood as the Robbin-Salamon index \cite{RS}. We refer to \cite{CFHW} for more detailed relations between them.

Let $\tilde{\gamma}$ be a contractible one-periodic orbit of the Hamiltonian $H(r,x)=f(r)$ which corresponds to $\abs{f'(r_0)}$-periodic Reeb orbit $\gamma$. The famous Viterbo $y$-intersection formula says that the action of $\tilde{\gamma}$ is
\begin{equation}\label{eq:intersection}
    \mathcal{A}_H(\tilde{\gamma})= f(r_0)-f'(r_0)r_0,
\end{equation}
which is the intersection point between the tangent line of $f(r)$ at $r_0$ and the $y$-axis. Note that since $W$ is an exact symplectic manifold, the action of a contractible orbit does not depend on the choice of cappings.

Now we are ready to study the complement of Liouville domains. Our proof is motivated by the Hartogs property of symplectic cohomology \cite{GV}, as well as the computation of spectral invariants for distance functions \cite{Ish}.

\begin{figure}
  \begin{tikzpicture}[yscale=0.8]
  \draw [->] (0,0)--(6,0);

  \draw [dotted] (2,4)--(2,0);
  \draw [dotted] (4,4)--(4,0);
  \node [below] at (2,0) {$1$};
  \node [below] at (4,0) {$1+\epsilon$};
 
  \node [right] at (6,0) {$r$};

  \draw (1.6,4.5) to [out=0, in=120] (2.2,4.2);
  \draw (0,4.5)--(1.6,4.5);
  \draw (2.2,4.2)--(3.55,1);
  \draw (4,0.4) to [out=300, in=180] (4.2,0.2);
  \draw [snake it] (4.2,0.2)--(5.5,0.2);

  \node [left] at (1,5) {$G_k$};

  \draw (0,3)--(1.6,3);
  \draw (1.6,3) to [out=0, in=120] (2.2,2.8);
  \draw (2.2,2.8)--(4,0.4);
  \node [left] at (1,3.5) {$F$};

  \end{tikzpicture}
  \caption{Hamiltonian functions in the collar region.}\label{fig:Ham}
\end{figure}
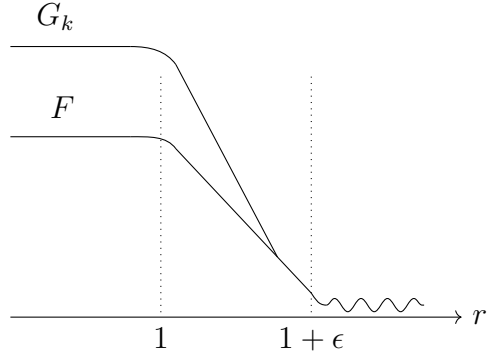

\begin{theorem}\label{t:Livoulle proof}
	Let $(M, \omega)$ be a closed rational symplectic manifold satisfying Condition $(\ref{eq:condition})$ with $\dim M =2n$. Suppose there is a Liouville domain $(K, \lambda)$ such that
	\begin{enumerate}
		\item $(K, d\lambda)$ is symplectically embedded in $(M, \omega)$ with $\pi_{1}(K)\to \pi_{1}(M)$ being injective;
		\item $\alpha:=\lambda\mid_{\partial K}$ is a non-degenerate contact form;
		\item $c_1(TK)\mid_{\pi_2(K)}=0$ and the Reeb orbits of $\lambda\mid_{\partial K}$ that are contractible in $K$ have Conley-Zehnder indices larger than $2-n$.
	\end{enumerate}
    Then $M-(K-\partial K)$ is a heavy set.
\end{theorem}
\begin{proof}
    We always identify $K$ with its image in $M$. By using the outward pointing Liouville flow near $\partial K$, we find a neighborhood 
    $$
    U_\epsilon:= K\cup_{\partial K} ([1,1+\epsilon)\times\partial K)
    $$ 
    of $K$ in $M$, for small $\epsilon>0$. The symplectic form on $[1,1+\epsilon)\times\partial K$ is $d(r\alpha)$ where $r$ is the collar coordinate. It suffices to show $M-U_\epsilon$ is heavy. We will construct Hamiltonian functions which are Morse functions on $K$ and $M-U_\epsilon$, and only depend on $r$ in the collar region, see Figure \ref{fig:Ham}. 

    We assume that the contact form $\alpha$ is non-degenerate. Hence its action spectrum is discrete, which we write as
    $$
    Spec(\alpha)=\{a_1< a_2<\cdots\}.
    $$

    Consider a smooth function $F$ on $M$ such that
    \begin{enumerate}
        \item $F$ is a positive $C^2$-small Morse function on $M-U_\epsilon$.
        \item $F$ only depends on $r$ in the collar region $\{1\leq r\leq 1+\epsilon\}$ with $F'(r)< 0$.
        \item When $r\in [1,1+\epsilon/3]$, $F(r)$ is concave.
        \item When $r\in [1+\epsilon/3,1+2\epsilon/3]$, $F(r)$ is linear with a slope in $(-a_2, -a_1)$.
        \item When $r\in [1+2\epsilon/3,1+\epsilon]$, $F(r)$ is convex.
        \item $F$ is a Morse function on $K$ with small first and second derivatives, and with no index-zero critical points.
    \end{enumerate}
    When we say a Morse function has small derivatives, we mean that its derivatives are small enough such that it does not have non-constant one-periodic orbits. Note that $K$ is a compact manifold with boundary, property $(6)$ can be achieved by \cite[Theorem 8.1]{Milnor}.

    Next for any positive integer $k$, we construct a smooth function $G_k$ such that
    \begin{enumerate}
        \item $G_k= F+k$ on $K$.
        \item $G_k=F$ on $M-U_\epsilon$ and on $\{1+2\epsilon/3\leq r\leq 1+\epsilon\}$.

        \item $G_k$ only depends on $r$ in the collar region $\{1\leq r\leq 1+\epsilon\}$ with $G_k'(r)< 0$;
        \item When $r\in [1,1+\epsilon/3]$, $G_k(r)$ is concave.
        \item When $r\in [1+\epsilon/3,1+\epsilon/2]$, $G_k(r)$ is linear with a slope not in $-Spec(\alpha)$.
        \item When $r\in [1+\epsilon/2,1+\epsilon]$, $G_k(r)$ is convex.

    \end{enumerate}
A pictorial depiction of $G_k$ is in Figure \ref{fig:Ham}.

Now we analyze the index and action of orbits of $F$ and $G_k$. The one-periodic orbits of $F$ or $G_k$ fall into three groups.
\begin{enumerate}
    \item Constant orbits in $M-U_\epsilon$.
    \item Constant orbits in $K$.
    \item Non-constant orbits in the collar region, corresponding to Reeb orbits of $\alpha$.
\end{enumerate}
Let $\gamma$ be a non-constant orbit of $F$ or $G_k$ that is contractible in $M$. By the incompressible condition on $K$, it is also contractible in $K$. Hence $\gamma$ admits a capping $u_I$ which is contained in $K$. Let's call such a capping an \textit{inner capping}. Since $K$ is an exact symplectic manifold with $c_1(TK)=0$, the index and action of $[\gamma,u_I]$ do not change among all inner cappings. A general capping $u$ of $\gamma$ is equivalent to $u_I+B$ for some $B\in\pi_2(M)$. All the constant orbits are non-degenerate, since we assume the function is Morse and with small derivatives. The contact form is non-degenerate, which makes the non-constant Hamiltonian orbits transversally non-degenerate. Hence we only need to break the $S^1$-symmetry of each orbit. We use the standard time-dependent perturbation in \cite[Proposition 2.2]{CFHW} to get a non-degenerate Hamiltonian $G^*_{k,t}$ out of $G_k$ and $F^*_t$ out of $F$. The perturbation can be suitably chosen such that
\begin{enumerate}
    \item $G^*_{k,t}=F^*_t$ on $M-U_\epsilon$ and on $\{1+2\epsilon/3\leq r\leq 1+\epsilon\}$.
    \item With respect to an inner capping, any one-periodic orbit of $G^*_{k,t}$ or $F^*_t$ in $U_\epsilon$ has Conley-Zehnder index larger than $-n$.
\end{enumerate}
Item $(1)$ can be achieved since $G_k$ and $F$ are equal on that region, so we just choose the same perturbation for them. Item $(2)$ can be achieved by the following reasons. By assumption, the Reeb orbits of $\lambda\mid_{\partial K}$ that are contractible in $K$ have Conley-Zehnder indices larger than $2-n$. When we consider such an orbit as a Hamiltonian orbit, with respect to an inner capping, the Conley-Zehnder index possibly changes by one, depending on whether $G_k(r)$ or $F(r)$ is concave or convex at this orbit. When we perturb it to break the $S^1$-symmetry, the Conley-Zehnder (Robbin-Salamon) index possibly changes by one \cite[Proposition 2.2]{CFHW}. For constant orbits in $K$, we don't perturb them since they are already non-degenerate. Particularly, their Morse indices are positive. Hence we get $(2)$. A direct consequence of $(2)$ is that any orbit in $U_\epsilon$, with respect to an inner capping, has a positive degree.

Now let $\gamma$ be a contractible orbit of $G^*_{k,t}$ which is in $K\cup \{1\leq r\leq 1+2\epsilon/3\}$. With respect to an inner capping, its action is larger than $\max F^*_t$. This is easy to see by the Viterbo $y$-intersection formula $(\ref{eq:intersection})$ and our perturbation is small enough. Suppose there is a general capping $u_I+B$ of $\gamma$ such that $\deg([\gamma,u_I+B])=0$. Then we have $c_1(TM)(B)<0$ since $\deg([\gamma,u_I])>0$. Set 
$$
\epsilon_M:= \inf\{\omega(B)>0\mid B\in \pi_2(M)\}.
$$ 
By the rational condition on $\omega$, we have $\epsilon_M>0$ and
$$
\mathcal{A}_{G^*_{k,t}}([\gamma,u_I+B])= \mathcal{A}_{G^*_{k,t}}([\gamma,u_I])+ \omega(B)\geq \mathcal{A}_{G^*_{k,t}}([\gamma,u_I])+\epsilon_M > \max F^*_t +\epsilon_M,
$$
by Condition $(\ref{eq:condition})$. In particular, we get
$$
Spec_0(G^*_{k,t})\cap (-\infty, \max F^*_t+\epsilon')= Spec_0(F^*_{t})\cap (-\infty, \max F^*_t+\epsilon')
$$ 
for some $\epsilon_M >\epsilon' >0$. The rational condition tells us that this is a discrete subset of $\RR$. So we can fix $\epsilon''>0$ small enough such that 
$$
(c(1,F^*_t)-\epsilon'', c(1,F^*_t)+\epsilon'')\cap Spec_0(F^*_{t})\cap (-\infty, \max F^*_t+\epsilon')= \{c(1,F^*_t)\}.
$$

Next we consider a sequence of functions
$$
F^*_t= F^*_{0,t}\leq F^*_{1,t}\leq \cdots \leq F^*_{l,t}=G^*_{k,t}
$$
such that
\begin{enumerate}
    \item $F^*_{i,t}$ is non-degenerate for all $i$.
    \item With respect to an inner capping, any one-periodic orbit of $F^*_{i,t}$ in $U_\epsilon$ has Conley-Zehnder index larger than $-n$.
    \item With respect to an inner capping, any one-periodic orbit of $F^*_{i,t}$ in $K\cup \{1\leq r\leq 1+2\epsilon/3\}$ has action larger than $\max F^*_t$. 
    \item $\int_{S^1}\max_M(F^*_{i+1,t}-F^*_{i,t})<\epsilon''/5$ for all $i$.
\end{enumerate}
The construction of $F^*_{i,t}$ is exactly the same as the construction for $G^*_{k,t}$. We interpolate from $F^*_t$ to $G^*_{k,t}$ by using a sequence of functions. The variation of adjacent functions is uniformly small. Since this is a monotone sequence, we have $F^*_{i,t}$ is independent of $i$ on the region where $F^*_t=G^*_{k,t}$. Item $(2)$ and $(3)$, together with the same action computation as for $G^*_{k,t}$, show that
$$
Spec_0(F^*_{i,t})\cap (-\infty, \max F^*_t+\epsilon')
$$ 
is independent of $i$. Item $(4)$, together with the Lipschitz property of spectral invariants show that $\abs{c(1,F^*_{0,t})-c(1,F^*_{1,t})}<\epsilon''/5$. Hence $c(1,F^*_{0,t})=c(1,F^*_{1,t})$ by the choice of $\epsilon''$ and the spectrality property. Repeat this process for $i$ inductively, we know that $c(1,F^*_{i,t})$ is independent of $i$. Particularly, we have $c(1,F^*_t)=c(1,G^*_{k,t})$.

Finally, pick any non-negative smooth function $H$ on $M$ that is zero on $M-U_\epsilon$, for any positive integer $m$, there exists some $k$ such that $G^*_{k,t}\geq mH$. Hence $c(1,mH)$ is uniformly bounded in $m$, which implies that $\mu(H)=0$. This shows that $M-U_\epsilon$ is heavy.

\end{proof}

\begin{remark}
    The above proof shows that if a non-negative Hamiltonian $H$ is supported in $K$ then $c(1,H)=0$. It would be interesting to explore whether the spectral invariants have certain locality property as in \cite{GT,T}.
\end{remark}

Now we give a technical enhancement of Theorem \ref{t:Livoulle proof}.

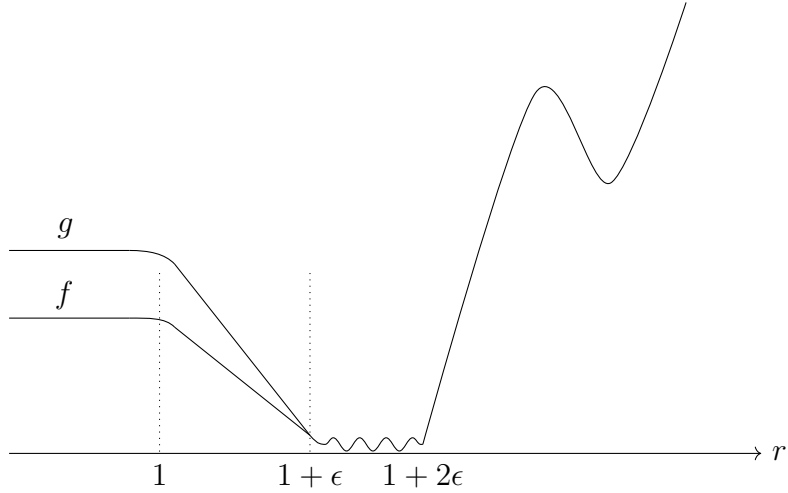
\begin{figure}
  \begin{tikzpicture}[yscale=0.6]
  \draw [->] (0,0)--(10,0);

  \draw [dotted] (2,4)--(2,0);
  \draw [dotted] (4,4)--(4,0);
  \node [below] at (2,0) {$1$};
  \node [below] at (4,0) {$1+\epsilon$};
  \node [below] at (5.5,0) {$1+2\epsilon$};

  \node [right] at (10,0) {$r$};

  \draw (1.6,4.5) to [out=0, in=120] (2.2,4.2);
  \draw (0,4.5)--(1.6,4.5);
  \draw (2.2,4.2)--(4,0.4);
  \draw (4,0.4) to [out=300, in=180] (4.2,0.2);
  \draw [snake it] (4.2,0.2)--(5.5,0.2);
  \draw plot [smooth] coordinates {(5.5,0.2) (7,8) (8,6) (9,10)};

  \node [left] at (1,5) {$g$};

  \draw (0,3)--(1.6,3);
  \draw (1.6,3) to [out=0, in=120] (2.2,2.8);
  \draw (2.2,2.8)--(4,0.4);
  \node [left] at (1,3.5) {$f$};

  \end{tikzpicture}
  \caption{Hamiltonian functions in the collar region.}\label{fig:Ham boundary}
\end{figure}

\begin{theorem}\label{t:boundary proof}
    Under the assumptions of Theorem \ref{t:Livoulle proof}, suppose further $K$ is heavy, then the set $\partial K$ is heavy.
\end{theorem}
\begin{proof}
    We continue to use the notations in Theorem \ref{t:Livoulle proof}. We first aim to prove that $\{1+\epsilon \leq r \leq 1+2\epsilon\}$ is heavy and then use continuity of spectral invariants to finish the proof.

    For small $\epsilon,\delta>0$, consider a smooth function $f$ on $M$ such that
    \begin{enumerate}
        \item $f$ is a positive $C^2$-small Morse function on $\{1+\epsilon\leq r\leq 1+2\epsilon\}$.
        \item $f$ only depends on $r$ in the collar region $\{1\leq r\leq 1+\epsilon\}$ with $f'(r)< 0$.
        \item When $r\in [1,1+\epsilon/3]$, $f(r)$ is concave.
        \item When $r\in [1+\epsilon/3,1+2\epsilon/3]$, $f(r)$ is linear with a slope not in $-Spec(\alpha)$.
        \item When $r\in [1+2\epsilon/3,1+\epsilon]$, $f(r)$ is convex.
        \item $f$ is a Morse function on $K$ with small first and second derivatives, and with no index-zero critical points, and $0<f<\delta<\epsilon_M/3$ on $U_\epsilon$.
        \item When $r\geq 1+2\epsilon$, $f$ is a positive function without other constraints.
    \end{enumerate}
    The construction of $f$ is the same as $F$ in Theorem \ref{t:Livoulle proof}. By $(6)$, there exists a non-negative function $\tilde{f}$ which is zero on $U_\epsilon$ and $\abs{f-\tilde{f}}<\delta$. We assume $K$ is heavy, hence $U_\epsilon$ is heavy. So we get $c(1,\tilde{f})=0$ and $0\leq c(1,f)<\delta$. Then we use time-dependent perturbation to get a non-degenerate $f^*_t$. The perturbation is chosen to be arbitrarily small such that  $-2\delta\leq c(1,f^*_t)<2\delta$. By the rational condition on $\omega$, we can fix some $\epsilon''>0$ such that
    $$
    (c(1,f^*_t)-\epsilon'', c(1,f^*_t)+\epsilon'')\subset (-2\delta,2\delta), \quad (c(1,f^*_t)-\epsilon'', c(1,f^*_t)+\epsilon'')\cap Spec_0(f^*_{t})= \{c(1,f^*_t)\}.
    $$

    Next, we slowly raise the level of $f$ in $U_\epsilon$, as in Theorem \ref{t:Livoulle proof}. Consider a function $g\geq f$ on $M$ such that
    \begin{enumerate}
        \item $g=f+C$ on $K$ for some positive $C$. 
        \item $g=f$ when $r\geq 1+2\epsilon/3$.
        \item $g$ only depends on $r$ in the collar region with $g'(r)< 0$.
        \item When $r\in [1,1+\epsilon/3]$, $g(r)$ is concave.
        \item When $r\in [1+\epsilon/3,1+2\epsilon/3]$, $g(r)$ is linear with a slope not in $-Spec(\alpha)$.
        \item When $r\in [1+2\epsilon/3,1+\epsilon]$, $g(r)$ is convex.
        \item $\int_{S^1}\max(g-f)<\epsilon''/5$.
    \end{enumerate}
    The construction of $g$ is the same as $G$ in Theorem \ref{t:Livoulle proof}. A pictorial depiction is in Figure \ref{fig:Ham boundary}. Then we perturb $g$ to get a non-degenerate $g^*_t$ with $\int_{S^1}\max(g^*_t-f^*_t)<\epsilon''/3$. On the region where $g=f$, we chose the same perturbation such that $f^*_t=g^*_t$. Note that an orbit $\gamma$ of $g^*_t$ which is not an orbit of $f^*_t$ is in $U_\epsilon$. The same index and action computation as in Theorem \ref{t:Livoulle proof} shows that if $\deg([\gamma,u])=0$ then its action is larger than $2\delta$. Recall that we chose $\delta<\epsilon_M/3$. Therefore by the Lipschitz property we get $c(1,g^*_t)=c(1,f^*_t)<2\delta$.

    Given a function $h$ on $M$ which is less than $f$ on $\{r\geq 1+\epsilon\}$, we can inductively raise the value of $g^*_t$ on $U_\epsilon$ to get a function $\tilde{g}^*_t\geq h$. The above procedure tells us $c(1,\tilde{g}^*_t)=c(1,f^*_t)$. Hence $c(1,h)\leq 2\delta$. On the other hand, the above argument only requires $f$ is a fixed positive function on $\{r\geq 1+2\epsilon\}$ and the estimate $c(1,g^*_t)<2\delta$ does not depend on this choice. Therefore given any function $h$ which is zero on $\{1+\epsilon\leq r\leq 1+2\epsilon\}$, it is bounded from above by such a $\tilde{g}^*_t$ and we have $c(1,h)<2\delta$. This shows the set $\{1+\epsilon\leq r\leq 1+2\epsilon\}$ is heavy. Since $\epsilon$ can be arbitrarily small, this shows $\partial K=\{r=1\}$ is heavy.

\end{proof}

Now we prove Theorem \ref{t:main lag}. 

\begin{proof}
    [Proof of Theorem \ref{t:main lag}]
    Let $L$ be an incompressible Lagrangian in $M$. Pick any Weinstein neighborhood $U$ of $L$. Without losing generality, we assume $U$ is induced by a non-degenerate Riemannian metric, which is a generic condition. Hence all of its closed geodesics are non-degenerate. When $\dim L\geq 3$, any Reeb orbit that is contractible in $U$ is also contractible in $\partial U$. Therefore its Conley-Zehnder index is non-negative, by Example \ref{ex:Reeb}, and we are in the situation of Theorem \ref{t:Livoulle proof}.
\end{proof}

Similarly, Theorem \ref{t:boundary} follows from Theorem \ref{t:boundary proof}.

\section{Lagrangian spheres}\label{sec:sphere}

Lagrangian spheres can appear in projective varieties as vanishing cycles in a degeneration process. The importance of this perspective in symplectic topology is noticed by Arnold \cite{Arnold} and Donaldson \cite{Don}. Here we recall two existence results by Seidel and Biran-Jerby.

\begin{theorem}[Corollary 4.4 in \cite{Seidel}]
    Any smooth complete intersection in $\CC P^n$ which is non-trivial (that is, not an intersection of hyperplanes) contains a Lagrangian sphere.
\end{theorem}

\begin{theorem}
    [Theorem 8.A in \cite{BJ}]
    Let $X\subset\CC P^n$ be an algebraic manifold and $\Sigma\subset X$ a hyperplane section. If the defect of $X$ is zero then $\Sigma$ contains a Lagrangian sphere.
\end{theorem}

Let us refer to \cite{BJ} for the definition of defect, and only mention that `most' algebraic manifolds have defect zero.

\begin{comment}

\subsection{Floer cohomology of incompressible Lagrangians}

Let $(M,\omega)$ be a symplectic manifold that is symplectic Calabi-Yau or negatively monotone. The Floer cohomology of a Lagrangian $L$ in $M$ is very hard to study. As far as we understand, there are three stages
\begin{enumerate}
    \item The Floer cohomology is not always well-defined. In \cite{FOOO}, an obstruction theory is established. If $L$ is unobstructed, then there is a Floer cohomology $HF(L)$, possibly with bounding cochains.
    \item When $HF(L)$ is defined, it could be zero. In \cite{FOOO}, a spectral sequence is constructed to compute $HF(L)$.
    \item If $HF(L)$ is defined and non-zero, possibly with bounding cochains, then $L$ is a heavy set \cite[Theorem 1.6]{FOOO2019}.
\end{enumerate}
The results in $(1)$ and $(2)$ are scattered around \cite{FOOO}. For reader's convenience, we collect some of them which fit our situations.

Suppose $L$ is incompressible. The homotopy exact sequence
\begin{equation}
    \cdots \to\pi_2(M)\to \pi_2(M,L)\to \pi_1(L)\to \pi_1(M)\to \cdots
\end{equation}
says the map $\pi_2(M)\to \pi_2(M,L)$ is surjective. Hence the symplectic area and the Maslov index of a disk bounding $L$ are determined by the value of $\omega$ and $c_1(TM)$ on $\pi_2(M)$. Therefore we get
\begin{enumerate}
    \item If $M$ is symplectic Calabi-Yau, then any $\beta\in \pi_2(M,L)$ has Maslov index $I(\beta)=0$.
    \item If $M$ is negatively monotone with constant $\tau<0$, then $I(\beta)= 2\tau\cdot\omega(\beta)$ for any $\beta\in \pi_2(M,L)$.
\end{enumerate}

\end{comment}

Now we discuss the Floer cohomology of Lagrangian spheres.

\begin{theorem}[Theorem 4.1 in \cite{W}]
    Let $(M,\omega)$ be a symplectic manifold with dimension larger than two. Suppose $M$ is symplectic Calabi-Yau or negatively monotone, and let $S$ be a Lagrangian sphere in $M$. For any $E>0$, there exists an open subset $U_E$ of compatible almost complex structures on $M$ such that $S$ does not bound any non-constant $J$-holomorphic disks with energy less than $E$ when $J\in U_E$.
\end{theorem}

\begin{corollary}\label{co:sphere}
    Any incompressible Lagrangian sphere in a closed symplectic manifold that is symplectic Calabi-Yau or negatively monotone has non-zero Floer cohomology. In particular, it is heavy.
\end{corollary}
\begin{proof}
    An incompressible Lagrangian circle has non-zero Floer cohomology, since it does not bound holomorphic disks by a topological reason. In higher dimensions, one can use the above theorem to define the Floer cohomology and show it is non-zero. Now we formally elaborate the construction via virtual techniques by Fukaya-Oh-Ohta-Ono \cite{FOOO, FOOO2020}. For a Lagrangian sphere in a general symplectic manifold, they have defined an $A_{\infty}$-structure $\{\mathfrak{m}_k\}$ on the de Rham complex $\Omega(S)$. If $\mathfrak{m}_1\circ\mathfrak{m}_1=0$, then they define the Floer cohomology $HF(S)$ as the homology of $(\Omega(S), \mathfrak{m}_1)$. In practice, the way to define $\mathfrak{m}_k$ is via an inductive process. Pick a sequence of energy bounds $E_1< E_2< \cdots \to +\infty$, for each $E_i$ they choose an almost complex structure $J_i$ to define $\mathfrak{m}_k^{i}$ by only considering $J_i$-holomorphic disks with boundary on $S$ and with energy less than $E_i$. Then the operator $\mathfrak{m}_k$ will finally be defined as a suitable homotopy limit of $\mathfrak{m}_k^{i}$. A nontechnical summary of this strategy is contained in \cite[Remark 22.26]{FOOO2020}. Back to our current situation, with the help of the above theorem by Welschinger, we choose $J_i$'s such that $S$ does not bound any non-constant $J_i$-holomorphic disks with energy less than $E_i$. This show that $\mathfrak{m}_1^{i}$ matches with the de Rham differential on $\Omega(S)$ for any $i$. After taking limit, we know $\mathfrak{m}_1$ is the de Rham differential and $HF(S)\neq 0$.

\end{proof}

\begin{remark}
    We add two remarks about the Floer cohomology of Lagrangian spheres via classical approaches.
    \begin{enumerate}
        \item Suppose that our symplectic manifold has a rational symplectic class $[\omega]\in H^2(M;\QQ)$. Then a Lagrangian sphere is rational in the sense of \cite[Definition 3.5]{CW}. Therefore, its Floer cohomology can be also defined via classical transversality method \cite{CM,CW}. We expect the above `no-holomorphic disk' theorem also implies that the Floer cohomology is non-zero.
        \item If our symplectic manifold is `very negative', then a Lagrangian sphere in it becomes strongly negative as in \cite[Definition 1.1]{La}. For example, a degree-$d$ hypersurface in $\CC P^n$ satisfies this condition when $d\geq 2n-1$. In this case, the Lazzarini structure theorem enables one to define $HF(S)$ and show it is non-zero; see \cite[Theorem D]{La}.
    \end{enumerate}

\end{remark}

\begin{proof}
    [Proof of Example \ref{ex:sphere}]
    When $n=1$, the hypersurface is a symplectic surface with genus larger than one. Hence the result follows from Usher \cite{Usher} and Kislev-Shelukhin \cite{KiS}. When $n\geq 3$, the existence results of Lagrangian spheres follow from the above work by Seidel and Biran-Jerby. Then we apply Theorem \ref{t:main} and Corollary \ref{co:sphere}. Actually, the theorems by Seidel and Biran-Jerby give many more examples in projective varieties beyond hypersurfaces.

    For $n=2$, we need to consider Lagrangian 2-spheres. Note that the sphere $S^2$ admits a non-degenerate Riemannian metric whose closed geodesics have Morse indices at least one. Hence it meets the assumptions of Theorem \ref{t:boundary proof} and Corollary \ref{co:sphere}.
\end{proof}

\begin{remark}
    In Section 6 of \cite{FOOO}, a spectral sequence is established to compute Lagrangian Floer cohomology for general Lagrangians. One should be able to use the computations therein to deduce more examples, especially for Lagrangians which are not spheres.
\end{remark}

\section{Super-heavy Lagrangian skeleton}

First let us recall the Biran-Giroux decomposition \cite{Biran,Gir}. For a closed integral symplectic manifold $(M^{2n},\omega)$, let $Y^{2n-2}$ be a closed symplectic hypersurface in $M$ whose Poincar\'e dual is $k[\omega]$ for some positive integer $k$. Then there is a symplectic decomposition
$$
(M, \omega)\cong (U, \omega_U)\sqcup L_Y
$$
where $U$ is a unit disk bundle over $Y$ with a standard symplectic form $\omega_U$, and $L_Y$ is the skeleton. This construction depends on further choices of metrics and connections; we refer to \cite[Section 2]{Biran} for more details. Fix a metric and let $\rho$ be the radius function on the disk bundle $U$. It induces a continuous map $\rho: M\to [0,1]$ such that $\rho^{-1}(0)=Y$, $\rho^{-1}(t)$ is the circle bundle of radius $t\in (0,1)$, and $\rho^{-1}(1)=L_Y$. More explicitly, the symplectic form on $\pi: U\to Y$ is 
\begin{equation}\label{eq:symp}
    \omega_U= \pi^*(\omega\mid_Y)+ \frac{1}{k}d(\rho^2\alpha),
\end{equation}
where $\alpha$ is a connection one-form on $U-Y$ with curvature $d\alpha= -\pi^*(\omega\mid_Y)$, see \cite{Biran}. Away from $Y$, we have $\omega= d((\rho^2/k -1)\alpha)$. From now on, we assume $k=1$ for notation simplicity. The proof for general $k$ follows from the same argument.

Set $R:=\rho^2$. It gives a collar coordinate on $M$, which is smooth when $R\in [0,1)$ and continuous at $R=1$. On $U-Y$ we have
$$
\omega= dR\wedge\alpha+ (R-1)d\alpha.
$$
For a function $H(R):U-Y\to \RR$ which only depends on $R$, its Hamiltonian vector field is
$$
X_H= -H'(R)\partial_{\alpha}
$$
where $\partial_{\alpha}$ is the vector field generating the rotation ($\RR/\ZZ$-action) on the fiber of the disk bundle with $\alpha(\partial_{\alpha})=1$. The closed orbits of $\partial_{\alpha}$ are multiple covers of circle fibers of the disk bundle. A primitive orbit has period one. Now consider a function $H(R):U\to \RR$ with
\begin{enumerate}
    \item $H'(R)>0$ and $H''(R)\leq 0$.
    \item $H(R)= \lambda R +C$ when $R\in [0,1/2)$ for some non-integer $\lambda$ and constant $C$.
    \item The set $\{R\mid H'(R)\in\ZZ\}$ is discrete.
\end{enumerate}
Then $H(R)$ is of Morse-Bott non-degenerate type. The one-periodic orbits of $H$ are
\begin{enumerate}
    \item Constant orbits along $Y$, with the Morse-Bott critical manifold being $Y$.
    \item Non-constant orbits on the slices $\{R\mid H'(R)\in\ZZ\}$, with the Morse-Bott critical manifold being the circle bundle over $Y$. Each non-constant orbit $\gamma$ is associate with an integer $w(\gamma)$, the \textit{multiple} of $\gamma$ over a circle fiber.
\end{enumerate}
Note that each of these orbits admit a natural capping $u_Y$ from the disk fiber. For constant orbits we take $u_Y$ to be constant cappings. We call them $Y$-cappings. Our situation is  the special case in \cite{BSV} when the divisor is smooth. Our $Y$-capping is written as $u_{out}$ there; see the paragraph after \cite[Proposition 1.15]{BSV}. For Morse-Bott non-degenerate orbits with cappings, one need to use the Robbin-Salamon index. However, we still write them as $CZ([\gamma,u_Y])$. The index is closely related to the multiple.

\begin{lemma}\label{l:index}
    With respect to the $Y$-cappings there hold
    \begin{enumerate}
        \item $\abs{CZ([\gamma,u_Y])+ 2w(\gamma)}\leq n$ for non-constant $\gamma$.
        \item $\abs{CZ([\gamma,u_Y])+ 2\lceil \lambda\rceil}\leq n$ for constant $\gamma$, where $\lceil \lambda\rceil$ is the ceiling function of $\lambda$.
    \end{enumerate}
\end{lemma}
\begin{proof}
    Since the orbits are multiple covers of some circle fiber, the computation follows from a local study. For example, see \cite[Lemma 4.22]{BSV}. The Hamiltonian function in \cite{BSV} has the opposite shape of ours: it achieves minimum on the skeleton and maximum on the hypersurface. Compare \cite[Figure 2]{BSV} with our Figure \ref{fig:Ham2}. Hence we have the opposite signs here. Our Floer theory conventions are the same as in \cite{BSV}.
\end{proof}

The actions are 

\begin{lemma}\label{l:action}
    With respect to the $Y$-cappings there hold
    \begin{enumerate}
        \item $\mathcal{A}_{H}([\gamma,u_Y])=H(0)$ for constant $\gamma$.
        \item $\mathcal{A}_{H}([\gamma,u_Y])= H(\gamma)- w(\gamma) R_0$ for non-constant $\gamma$ lying in the slice $R=R_0$.
    \end{enumerate}
\end{lemma}
\begin{proof}
The $Y$-capping of a non-constant orbit $\gamma$ is a disk covering the disk fiber $w(\gamma)$-times, with the opposite orientation since our $X_H$ is in the opposite direction of $\partial_\alpha$. Hence $\omega(u_Y)=- w(\gamma) R_0$.
\end{proof}

Now we prove Theorem \ref{t:skeleton}. The argument is similar to Theorem \ref{t:Livoulle proof}.

\begin{figure}
  \begin{tikzpicture}[yscale=0.8]
  \draw [<-] (0,4)--(6,4);

  \draw [dotted] (2,4)--(2,-2);
  \draw [dotted] (3,4)--(3,-2);
  \draw [dotted] (6,4)--(6,-2);
  \draw [dotted] (4,4)--(4,-2);
  \node [above] at (2,4) {$b$};
  \node [above] at (6,4) {$0$};

  \node [above] at (4,4) {$b-\epsilon$};
  \node [below] at (3,-2) {$b-\epsilon/2$};
  \node [left] at (0,4) {$R$};

  \draw [snake it] (0,3)--(2,3);
  \draw (2,3) to [out=0] (4,2);
  \draw (4,2)--(6,0);
  \draw (3.5,2.4)--(6,-2);
  \node [right] at (6.2,0) {$F$};
  \node [right] at (6.2,-2) {$G$};

  \end{tikzpicture}
  \caption{Hamiltonian functions in the collar region.}\label{fig:Ham2}
\end{figure}
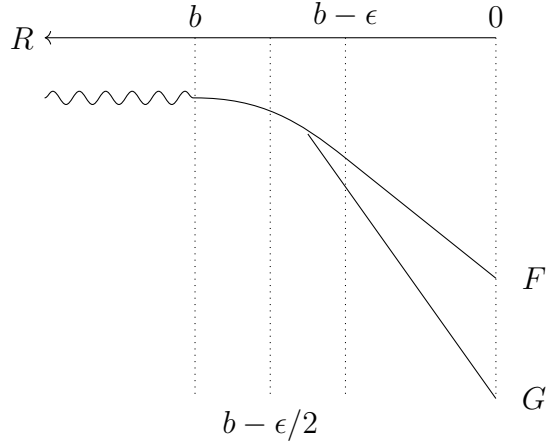

\begin{proof}[Proof of Theorem \ref{t:skeleton}]
    As mentioned above, we only prove the $k=1$ case.

    Set $W_b:=\{R\geq b\}\subset M$. It suffices to show $W_b$ is super-heavy for any $b\in(0,1)$. Fix $b$ and $\epsilon$ with $b-\epsilon>0$. Pick any non-positive function $H$ on $M$ that is zero on $W_{b-\epsilon}$. For any positive integer $m$, the function $mH$ is bounded from below by some $G$ which has the shape in Figure \ref{fig:Ham2}. In the following we will construct $G$ and show $c(1,G)$ is independent of $m$. Therefore $c(1,mH)$ is uniformly bounded and $\mu(H)=0$, which proves $W_{b-\epsilon}$ is super-heavy. The particular value of $b$ will not be used in the proof, this shows that $W_b$ is super-heavy for any $b\in(0,1)$. 

    Let $F:M\to \RR$ be a smooth function such that
    \begin{enumerate}
        \item $F$ only depends on $R$ when $R\in (0,b)$ and $F'(R)>0, F''(R)\leq 0$.
        \item The set $\{R\in(0,b)\mid F'(R)\in\ZZ\}$ is discrete.
        \item $F$ is a negative $C^2$-small Morse function when $R\geq b$.
        \item $F$ is linear in $R$ when $R\in (0,b-\epsilon/2)$, with a slope in $(10n, 10n+1)$.
    \end{enumerate}
    Then we also construct another function $G$ such that
    \begin{enumerate}
    \item $G$ only depends on $R$ when $R\in (0,b)$ and $G'(R)>0, G''(R)\leq 0$.
    \item $G=F$ when $R\in (b-\epsilon/2, 1]$.
    \item $G$ is linear in $R$ when $R\in (0,b-\epsilon)$, with a slope in $(10n+1, 10n+2)$.
    \end{enumerate}
Pictorial depiction of $F,G$ is in Figure \ref{fig:Ham2}. Both $F$ and $G$ are Morse-Bott non-degenerate. We apply time-dependent perturbations to make them non-degenerate. The resulting $F^*_t,G^*_t$ can be made that
\begin{enumerate}
    \item $G^*_t=F^*_t$ when $R\in (b-\epsilon/2, 1)$.
    \item For any capped orbit $[\gamma,u]$ of $F^*_t$ or $G^*_t$ it corresponds to a capped orbit $[\gamma',u']$ of $F$ or $G$ such that
    \begin{enumerate}
        \item $\abs{CZ([\gamma,u])-CZ([\gamma',u'])}\leq n$.
        \item $\abs{\mathcal{A}([\gamma,u])-\mathcal{A}([\gamma',u'])}$ is arbitrarily small.
    \end{enumerate} 
\end{enumerate}
This is essentially a higher dimensional version of the perturbation we used in Theorem \ref{t:Livoulle proof} to break the $S^1$-symmetry. Now the critical manifolds are $Y$ and the slices where $F'(R)\in\ZZ$, which we assume to be disjoint. Hence we can add on time-dependent perturbations locally. These perturbations are small Morse functions on the critical manifolds. After perturbation the non-degenerate orbits correspond to critical points of the Morse functions. The index and action change are computed in \cite[Lemma 4.17]{BSV}. We refer to it and the references therein for more details. 

Since $G^*_t=F^*_t$ when $R\in (b-\epsilon/2, 1)$, they share common orbits in this region. Let $\gamma$ be an orbit of $G^*_t$ in $\{0\leq R\leq b-\epsilon/2\}$. It comes from a perturbation of $\gamma'$ which is an orbit of $G$ in the same region. So $\gamma'$ is either a constant orbit on $Y$, or a non-constant orbit with $w(\gamma')=10n+1$. First we consider the case $\gamma'$ is non-constant. Choosing a $Y$-capping for $\gamma'$, we have $CZ([\gamma',u_Y])\leq -19n -2$ by Lemma \ref{l:index}. Hence $\gamma$ with a suitably chosen capping $u$ has $CZ([\gamma,u])\leq -18n -2$. The orbit $\gamma'$ lies in some slice $R=R_0$ with $R_0\in (b-\epsilon,b-\epsilon/2)$. By Lemma \ref{l:action}, we have 
$$
\mathcal{A}_G([\gamma',u_Y])= G(\gamma')- (10n+1) R_0\leq F(R_0)- (10n+1) R_0< \min F^*_t.
$$
The first inequality follows from $F\geq G$. The second inequality follows from a computation of $\min F$, given that $F$ is linear on $R\in (0,b-\epsilon/2)$ with a slope in $(10n, 10n+1)$. Since our perturbation is small, we still have that 
$$
\mathcal{A}_{G^*_t}([\gamma,u])< \min F^*_t.
$$
The same holds for a constant orbit $\gamma'$ because $F(0)>G(0)$. Hence we know that any orbit $\gamma$ of $G^*_t$ in $\{0\leq R\leq b-\epsilon/2\}$ has a capping $u$ such that
$$
\deg([\gamma,u])\leq -17n -2, \quad \mathcal{A}_{G_t}([\gamma,u])< \min F^*_t.
$$
For a general capping $u+B$ with $B\in \pi_2(M)$, if $\deg([\gamma,u+B])=0$ then $c_1(TM)(B)>0$ which implies that
$$
\mathcal{A}_{G^*_t}([\gamma,u+B])\leq \mathcal{A}_{G^*_t}([\gamma,u])-\epsilon_M< \min F^*_t -\epsilon_M,
$$
by Condition \ref{eq:condition}. Here $\epsilon_M >0$ is the rational constant defined in Theorem \ref{t:Livoulle proof}. In particular, we get
$$
Spec_0(G^*_{t})\cap (\min F^*_t -\epsilon', +\infty)= Spec_0(F^*_{t})\cap (\min F^*_t -\epsilon', +\infty)
$$
for some $0<\epsilon'<\epsilon_M$. It is a discrete subset of $\RR$ by our integral assumption. The rest of the proof is similar to Theorem \ref{t:Livoulle proof}. We interpolate from $F^*_t$ to $G^*_t$ by using a sequence of functions, with variations between adjacent ones being uniformly small. The Lipschitz property of spectral invariants give us that $c(1,F^*_t)=c(1,G^*_t)$. Then we can inductively repeat these steps to make the slope of $G$ goes to infinity. The interpolation process will show $c(1,G)$ stays the same.

\end{proof}

\begin{remark}\label{re:comp}
    One can see that the computation above only happens on $U$ and does not depend on the complement of $U$. What we actually showed is that if $U$ is an open subset in a rational $(M,\omega)$ satisfying Condition \ref{eq:condition} and the induced symplectic structure on $U$ is of the form (\ref{eq:symp}), then $M-U$ is super-heavy. In the Biran decomposition \cite[Theorem 2.6.A]{Biran}, the skeleton $\Delta_\mathcal{P}$ is defined via a polarization $\mathcal{P}$. Different polarizations give possibly different $\Delta_\mathcal{P}$. On the other hand, the complement of $\Delta_\mathcal{P}$ is symplectomorphic to the standard symplectic disk bundle for any polarization.
\end{remark}

We expect that by a more careful analysis as in \cite{BSV}, the above theorem could be extended to a symplectic manifold satisfying Condition \ref{eq:condition} with a simple normal crossings divisor.

\section*{Acknowledgements}
We benefited from conversations with Pierre-Alexandre Mailhot, Egor Shelukhin, Yusuke Kawamoto, Yoel Groman, Cheuk Yu Mak, Umut Varolgunes, Yanki Lekili, Noah Porcelli and Hang Yuan. We thank the referee for very helpful comments. The author is supported by EPSRC grant EP/W015889/1.

\bibliography{main}

\bibliographystyle{amsplain}

\begin{comment}

\end{comment}

\end{document}